\documentclass[11pt]{ip-journal} 

\usepackage{wasysym}
\usepackage{cancel}
\usepackage{latexsym, amssymb, amsmath}
\usepackage{soul,esint}
\usepackage{amsfonts, graphicx}
\usepackage{graphicx,color}
\newcommand{\be}{\begin{equation}}
\newcommand{\ee}{\end{equation}}
\newcommand{\beq}{\begin{eqnarray}}
\newcommand{\eeq}{\end{eqnarray}}

\newtheorem{thm}{Theorem}[section]

\newtheorem{lma}[thm]{Lemma}
\newtheorem{prop}[thm]{Proposition}

\theoremstyle{remark}
\newtheorem{rem}[thm]{Remark}
\numberwithin{equation}{section}

\def\be{\begin{equation}}
\def\ee{\end{equation}}
\def\bee{\begin{equation*}}
\def\eee{\end{equation*}}

\def\lf{\left}
\def\ri{\right}

\def\Ric{\text{\rm Ric}}
\def\Rm{\text{\rm Rm}}

\def\p{\partial}

\def\rheat{\lf(\frac{\p}{\p t}-\Delta_{g(t)}\ri)}

\def\e{\varepsilon}
\def\a{{\alpha}}
\def\b{{\beta}}

\def\R{\mathrm{R}}\def\cR{\mathcal{R}}

\begin{document}
\title[]
{Three-manifolds with non-negatively pinched Ricci curvature}

\author{Man-Chun Lee}
\address[Man-Chun Lee]{Department of Mathematics, The Chinese University of Hong Kong, Shatin, N.T., Hong Kong
}
\email{mclee@math.cuhk.edu.hk}

\author{Peter M. Topping}
\address[Peter M. Topping]{Mathematics Institute, Zeeman Building, University of Warwick, Coventry CV4 7AL}
 \email{P.M.Topping@warwick.ac.uk}


\date{10 September 2022}

\begin{abstract}
We show that every complete non-compact three-manifold with non-negatively pinched Ricci curvature admits a complete Ricci flow solution for all positive time, with scale-invariant curvature decay and preservation of pinching. 
Combining with recent work of Lott and Deruelle-Schulze-Simon 
gives a proof of Hamilton's  pinching conjecture without additional hypotheses.
\end{abstract}

\keywords{}

\maketitle{}

\markboth{}{Three-manifolds with pinched Ricci curvature}
\section{Introduction}

Suppose $(M,g)$ is a smooth complete $n$-dimensional Riemannian manifold with a positive lower Ricci curvature bound $\Ric(g)\geq (n-1)r^{-2}>0$. 
The classical theorem of Bonnet-Myers tells us that the diameter of $(M,g)$ is then bounded above by $\pi r$, which is sharp on any round sphere.

In this paper we consider a situation in which the Ricci curvature is still non-negative, but does not have a uniform positive lower bound. Instead it is assumed to be controlled from below in terms of the scalar curvature $\cR$. 
Our goal is to extend recent work of Lott \cite{Lott2019} and Deruelle-Schulze-Simon \cite{DSS22} in order to prove 
Hamilton's pinching conjecture \cite[Conjecture 3.39]{CLN09}:
\begin{thm}[Hamilton's pinching conjecture]
\label{main_thm}
Suppose $(M^3,g_0)$ is a complete (connected) three-dimensional Riemannian manifold with $\Ric\geq \e \cR\geq 0$ for some $\e>0$. Then  $(M^3,g_0)$ is either flat or is compact.
\end{thm}


This conjecture is the intrinsic analogue of a theorem of Hamilton \cite{Hamilton1994}
in which $\Ric$ and $\cR$ are replaced by the second fundamental form and mean curvature of a hypersurface of Euclidean space, respectively.
Partial results towards the conjecture are available in the presence of additional upper and lower curvature bound assumptions. 
B.-L. Chen and X.-P. Zhu \cite{ChenZhu2000} proved that if one assumes additionally that 
\begin{enumerate}
\item
$|\Rm(g_0)|\leq \Lambda$ for some $\Lambda>0$ and
\item
$\mathrm{K}(g_0)\geq 0$, i.e. all sectional curvatures are nonnegative,
\end{enumerate}
then the conjecture is true. Later,  J. Lott \cite{Lott2019} presented an alternative method to prove the conjecture under the weaker additional assumptions that
\begin{enumerate}
\item
$|\Rm(g_0)|\leq \Lambda$ for some $\Lambda>0$ and
\item
$\mathrm{K}(g_0)\geq -\frac{C_0}{d^2_{g_0}(\cdot,p))}$ for some $C_0>0$ and $p\in M$, i.e. 
the sectional curvatures have an inverse quadratic lower bound.
\end{enumerate}
The weakening of the lower sectional curvature bound moving from Chen-Zhu's work to Lott's work is significant because it blocks the use of Hamilton's Harnack estimate, and so new ideas are required, some of which are employed in the final resolution of Hamilton's conjecture. 
One notable consequence of Lott's work is that a complete, non-compact 
three-manifold 
satisfying the Ricci pinching condition $\Ric\geq \e \cR\geq 0$ (for some $\e>0$)
and with bounded curvature is found either to be flat or to have cubic volume growth.

Very recently, using Lott's work as a foundation,
Deruelle-Schulze-Simon \cite{DSS22} have made the decisive step of dramatically weakening the additional hypotheses to only require additionally that
\begin{enumerate}
\item
$|\Rm(g_0)|\leq \Lambda$ for some $\Lambda>0$.
\end{enumerate}

The purpose of this paper is to demonstrate that hypothesis (1) can be removed, giving the  result stated in Theorem \ref{main_thm}.

All of the prior works on this conjecture use Ricci flow.  As part of his approach, Lott proved that under his additional bounded curvature hypothesis 
one can solve the Ricci flow for all time with scale-invariant curvature decay and while preserving the pinching. In this paper we demonstrate that this is possible without requiring the bounded curvature assumption (1). 
Precisely, we prove:
\begin{thm}\label{SecIntro-RFexist}
Suppose $(M^3,g_0)$ is a complete non-compact three-dimensional Riemannian manifold with $\Ric(g_0)\geq \e \cR(g_0)\geq 0$ for some 
$\e>0$. 
Then there exists $a=a(\e)>0$ such that the Ricci flow has a complete long-time solution $g(t)$ starting from $g_0$ so that
\begin{enumerate}
\item[(a)] $\Ric(g(t))\geq \e \cR(g(t))\geq 0$;
\item [(b)] $|\Rm(g(t))|\leq at^{-1}$;
\end{enumerate}
for all $(x,t)\in M\times (0,+\infty)$. 
\end{thm}
\begin{rem}
Prior to this result, even short-time existence in this situation was unknown. 
In the absence of the pinching constraint, short-time existence 
\emph{remains} an open problem. That is, given a complete three-manifold with nonnegative Ricci curvature, but no assumed overall curvature bound, it is uncertain as to whether there necessarily exists some Ricci flow continuation. Certainly the $a/t$ curvature decay cannot be true 
in general without the pinching hypothesis.
\end{rem}

Although the proof of Theorem \ref{SecIntro-RFexist} is independent of the ideas used by Lott to prove the bounded curvature case, there is one crucial additional property of Lott's solution that is automatically inherited by our solution from his work, namely that if there exists at least one point in space-time where the scalar curvature is positive then the flow has \emph{positive} asymptotic volume ratio $v_0>0$ at each time.
In practice, instead of trying to substitute our existence theory in place of the analogous existence result of Lott in the development of the theory, we can apply our existence theorem to reduce the full Hamilton pinching conjecture to the situation that has already been addressed, as follows.
\begin{proof}[{Proof of Theorem \ref{main_thm}}]
Suppose that $M^3$ is not compact. Then Theorem \ref{SecIntro-RFexist} applies
to give a complete Ricci flow evolution $g(t)$ that preserves the pinching condition and immediately has bounded curvature. Thus for each $t_0>0$ the manifolds 
$(M^3, g(t_0))$ satisfy the hypotheses of 
Deruelle-Schulze-Simon \cite[Theorem 1.3]{DSS22} and we can deduce that 
$g(t_0)$ is flat for each $t_0>0$, and the Ricci flow is independent of time. 
In particular, the initial data $(M^3,g_0)$ is flat and 
Theorem \ref{main_thm} follows.
\end{proof}
Thus the significance of our contribution is to extend the applicability of earlier work.

\vskip 0.2cm

\noindent
\emph{Acknowledgements:} PT was supported by EPSRC grant EP/T019824/1 and
thanks Henry Popkin for conversations about the classical work on this topic.
For the purpose of open access, the authors have applied a Creative Commons Attribution (CC BY) licence to any author accepted manuscript version arising.


\section{Strategy of the proof}

The key step to proving the global existence of Theorem \ref{SecIntro-RFexist}
will be to obtain local existence, with estimates, for a \emph{uniform} time:

\begin{thm}[{Special case of Theorem \ref{Thm:Construct-partialRF}}]
\label{Thm:Construct-partialRF_special}
For all $\e\in (0,\frac1{12})$, there exist $T(\e),a_0(\e)>0$ such that the following holds.
Suppose $(M^3,g_0)$ is a complete non-compact three-dimensional Riemannian manifold so that
$$\Ric(g_0)\geq \e \cR(g_0)\geq 0.$$
Then for any $p\in M$, there exists a smooth Ricci flow solution $g(t)$ defined on $B_{g_0}(p,1)\times [0,T]$ such that $g(0)=g_0$, $|\Rm(x,t)|\leq a_0t^{-1}$ and 
\begin{equation}
\Ric(x,t)\geq \e\cR(x,t)-1
\end{equation}
for all $(x,t)\in B_{g_0}(p,1)\times (0,T]$.
\end{thm}
In order to prove global  existence (in space and time)
we will exploit the scale invariance of the hypotheses and apply the local existence theorem not to $g_0$ but to a sequence of blown down metrics $R_i^{-2}g_0$ with $R_i\to\infty$. 
The resulting Ricci flows can be parabolically scaled back to give
Ricci flows $g_i(t)$ on $B_{g_0}(p,R_i)\times [0,T R_i^2]$, 
still with the same $a_0t^{-1}$ curvature decay, and now with $g_i(0)=g_0$ where defined. By parabolic regularity, i.e. by Shi's estimates, it will be possible to take a subsequence and extract  a smooth complete limit Ricci flow. Moreover, an effect of the rescaling will be that this limit flow will be Ricci pinched, without error, and thus give the flow required by 
Theorem \ref{SecIntro-RFexist}. 

The proof, therefore, will rest on local existence as in  Theorem \ref{Thm:Construct-partialRF_special}.
Obtaining local existence over a time interval that can depend directly on $g_0$ is straightforward. One can modify the metric conformally in a boundary layer around the local region in order to make it complete and of bounded curvature, and then flow using Shi's classical existence theorem from \cite{Shi1989}, cf. 
\cite{JEMS, Hochard2016, SimonTopping2021}. In order to construct a flow whose existence time has no dependence on $g_0$ other than a dependence on the pinching constant $\e$, we develop an approach from \cite{Hochard2016} and \cite{SimonTopping2021} that essentially allows us to keep restarting the flow an uncontrollably large number of times until we have the flow over a uniform time interval. 
In order to make this work we need a collection of \emph{a priori} estimates that we compile in Section \ref{apriorisect}. These quantify how Ricci pinching, curvature decay and roundness are related to each other under Ricci flow. For example, Lemma \ref{SecEst-pinching-curvatureEstimate} tells us that a local Ricci flow that becomes round where the curvature blows up must satisfy $a_1/t$ curvature decay.

\section{A priori estimates for Ricci flow}
\label{apriorisect}
In this section, we will derive local estimates for the Ricci flow on three-manifolds satisfying a pinching assumption. First, we have the following local persistence of Ricci pinching.

\begin{lma}\label{SecEst-almost-pinching}
Suppose $(M^3,g(t)),t\in [0,T]$ is a smooth solution to the Ricci flow such that for some $x_0\in M$, we have $B_0(x_0,1)\Subset M$ for $t\in [0,T]$. 
If there exist $a>0$ and $\frac1{12}>\e>0$ such that
\begin{enumerate}
\item[(i)] $\Ric(g(0))\geq \e \cR(g(0))\geq 0$ on $B_{0}(x_0,1)$;
\item[(ii)] $|\Rm(g(t))|\leq at^{-1}$ on $B_0(x_0,1)$, $t\in (0,T]$;
\end{enumerate}
then there exists $S_1(a)>0$ such that for all $t\in [0,T\wedge S_1]$ 
we have 
$$\Ric(x_0,t)\geq \e \cR(x_0,t)-1.$$
In particular, $\cR(x_0,t)>-4$ for 
$t\in [0,T\wedge S_1]$.
\end{lma}

It is convenient to obtain the lower scalar curvature bound here by tracing the pinching estimate, but we remark that the property of having a local lower bound on the scalar curvature is
always preserved in the presence of $a/t$ curvature decay owing to an estimate of B.-L. Chen; see  \cite{Chen2009} and \cite[Section 8]{SimonTopping2016}.

\begin{proof}
For $(x,t)\in M\times [0,T]$, define
\begin{equation}
\lambda(x,t)=\inf\{s\geq 0:\Ric(x,t)-\e \cR(x,t)g(x,t)+sg(x,t) >0 \}.
\end{equation}

Clearly, we have
\begin{enumerate}
\item[(a)]$\lambda(x,0)=0$ for $x\in B_0(x_0,1)$;
\item[(b)]$\lambda(x,t)\leq Cat^{-1}$ for $t\in (0,T]$ and $x\in B_0(x_0,1)$, with $C$ universal.
\end{enumerate}

It suffices to estimate $\lambda(x_0,t)$ from above for $t$ small. 
By \cite[(74)-(75)]{ChenZhu2000}, $\lambda(x,t)$ satisfies
\begin{equation}
\rheat \lambda\leq h\lambda -A(\lambda_2^2+\lambda_3^2)+2B\lambda_2\lambda_3
\end{equation}
in the barrier sense, where $\lambda_i$ are eigenvalues of $\Ric(x,t)$ with respect to $g(t)$ with $\lambda_1\leq \lambda_2\leq \lambda_3$, $h$ is some linear combination of $\lambda_i$, $A=1-2\e+\frac{\e}{1-\e}-2\e \left(\frac{\e}{1-\e}\right)^2$ and 
$B=1-\frac{\e}{1-\e}+2\e\left(\frac{\e}{1-\e}\right)^2$. 
Because $0<\e<\frac1{12}$, we have
$A\geq B\geq 0$ and hence $\lambda$ satisfies
\begin{equation}
\rheat \lambda\leq C_0|\Rm| \lambda
\end{equation}
for some constant $C_0>0$. 
We may now apply \cite[Theorem 1.1]{LeeTam2020} to conclude. 
This completes the proof.
\end{proof}

The next lemma shows that under almost Ricci pinching, the Ricci flow will be almost Einstein.

\begin{lma}\label{SecEst-pinching-Implies-spherical}
Suppose $(M^3,g(t)),t\in [0,T]$ is a smooth solution to the Ricci flow 
and that for some $x_0\in M$, $a>0$ and $\frac1{100}>\e>0$ we have
\begin{enumerate}
\item[(i)]  $B_t(x_0,1)\Subset M$ for all $t\in [0,T]$;
\item[(ii)] $\Ric(g(t))\geq \e \cR(g(t))-1$ for all $x\in B_t(x_0,1)$ and $t\in [0,T]$;
\item[(iii)] $|\Rm(x,t)|\leq at^{-1}$ for all $x\in B_t(x_0,1)$ and $t\in (0,T]$.
\end{enumerate}
Then there exist $S_2=S_2(a,\e)$, $L_1(\e)>0$ and $\sigma(\e)\in (1,2)$ so that 
at $x_0$, for all $t\in (0,T\wedge S_2]$, we have 
$$\left|\Ric-\frac13 \cR g\right|^2\leq \frac{L_1}{t^{2-\sigma}} \cdot (\cR+4)^\sigma.$$
\end{lma}

\begin{proof}
By increasing $a$ if necessary, we may assume that $a>1$.
In the proof, we will use $c_i$ and $C_i$ to denote constants depending only on $\e,\sigma$. 
We will follow closely the computations of Hamilton \cite{Hamilton1982} and
Lott \cite{Lott2019}. We may assume $S_2\leq 1$ from the outset, and will constrain 
$S_2$ further during the proof.

By taking the trace of (ii), we have
\begin{equation}\label{LemmaStep: almostpreserveRRic}
\cR(g(t))+4>\displaystyle \frac12
\end{equation}
on $B_t(x_0,1)$. In what follows, we will always work on $B_t(x_0,1)$ where \eqref{LemmaStep: almostpreserveRRic} holds. Consider the function
$$f(x,t)=\frac{1}{(\cR+4)^{\sigma}}\left(|\Ric|^2-\frac13 (\cR+4)^2 \right)$$
on $B_t(x_0,1)$ where $1<\sigma<2$ is a constant to be chosen later.  This is well-defined on  on $B_t(x_0,1)$ thanks to  \eqref{LemmaStep: almostpreserveRRic}. This is a 
modified version of the quantity considered in \cite[Theorem 10.1]{Hamilton1982}. We will later show that a suitable bound on $f(x,t)$ is equivalent to the conclusion of the lemma.


We first compute the evolution of $|\Ric|^2 (\cR+4)^{-\sigma}$.
By \cite[Lemma 10.2]{Hamilton1982}, if $\lambda\geq \mu\geq \nu$ denote the eigenvalues of $\Ric$ with respect to $g(t)$, then
\begin{equation}
\rheat  |\Ric|^2= -2|\nabla_i \R_{jk}|^2+4(T-L)
\end{equation}
where
\begin{equation}
\left\{
\begin{aligned}
T&=\lambda^3+\mu^3+\nu^3; \\
L&=\lambda^3+\mu^3+\nu^3-\left( \lambda^2\mu+\lambda\mu^2+\lambda\nu^2+\lambda^2\nu+\mu^2\nu+\mu\nu^2\right)+3\lambda\mu\nu.
\end{aligned}
\right.
\end{equation}
%
%
We will use this to compute the evolution of $f$ by following \cite[Lemma 10.3]{Hamilton1982} as closely as possible to give
\begin{equation}
\begin{split}
&\quad \rheat \frac{|\Ric|^2}{(\cR+4)^{\sigma}}\\
&=\frac{1}{(\cR+4)^{\sigma}}\rheat |\Ric|^2 +|\Ric|^2\rheat (\cR+4)^{-\sigma}\\
&\quad -2\langle \nabla |\Ric|^2, \nabla (\cR+4)^{-\sigma}\rangle\\
&=(\cR+4)^{-\sigma}\left[-2|\nabla \Ric|^2+4(T-L) \right]-2\sigma|\Ric|^4 (\cR+4)^{-\sigma-1}\\
&\quad -\sigma(\sigma+1)(\cR+4)^{-\sigma-2}|\nabla \cR |^2|\Ric|^2 +2 \sigma (\cR+4)^{-\sigma-1}\langle\nabla |\Ric|^2 , \nabla \cR\rangle\\
&=(\cR+4)^{-\sigma-2}\Big\{-2 (\cR+4)^2 |\nabla \Ric|^2+2 \sigma (\cR+4)\langle \nabla |\Ric|^2, \nabla \cR\rangle\\
&\qquad\qquad\qquad\qquad -\sigma(\sigma+1)  |\nabla \cR|^2 |\Ric|^2 \Big\}\\
&\quad+ (\cR+4)^{-\sigma-1}\left[4(T-L)(\cR+4)-2\sigma |\Ric|^4 \right].
\end{split}
\end{equation}
%
We now handle the first bracket more carefully. Following the computation in \cite[page 284]{Hamilton1982}, we see that 
\begin{equation}
\begin{split}
&-2(\cR+4)^2|\nabla \Ric|^2+2\sigma (\cR+4)\langle  \nabla \cR,\nabla |\Ric|^2\rangle-\sigma (\sigma+1) |\Ric|^2 |\nabla \cR|^2\\
&=-2|(\cR+4)\nabla_i \R_{jk}- \cR_i\R_{jk}|^2+2(\sigma-1) (\cR+4) \langle  \nabla \cR,\nabla |\Ric|^2\rangle\\
&\quad +[2-\sigma (\sigma+1)] |\Ric|^2 |\nabla \cR|^2\\
&=-2|(\cR+4)\cdot \nabla_i \R_{jk}- \cR_i \R_{jk}|^2+[2-\sigma (\sigma+1)] |\Ric|^2 |\nabla \cR|^2\\
&\quad +2(\sigma-1) (\cR+4)^{\sigma+1}\left( \left\langle \nabla \cR, \nabla \frac{|\Ric|^2}{(\cR+4)^{\sigma}}\right\rangle +\frac{\sigma |\nabla \cR|^2 |\Ric|^2 }{(\cR+4)^{\sigma+1}}\right)\\
&=-2|(\cR+4)\cdot \nabla_i \R_{jk}- \cR_i  \R_{jk}|^2+(\sigma-2)(\sigma-1)|\Ric|^2 |\nabla \cR|^2\\
&\quad +2(\sigma-1) (\cR+4)^{\sigma+1}\left \langle \nabla \cR, \nabla \frac{|\Ric|^2}{(\cR+4)^{\sigma}} \right\rangle.
\end{split}
\end{equation}
Hence,
\begin{equation}
\begin{split}
& \rheat  \frac{|\Ric|^2}{(\cR+4)^{\sigma}}\\
&= -\frac{2}{(\cR+4)^{\sigma+2}}|(\cR+4)\cdot \nabla_i \R_{jk}- \cR_i  \R_{jk}|^2\\
&\quad +\frac{2(\sigma-1)}{\cR+4}\left\langle\nabla \cR, \nabla \frac{|\Ric|^2}{(\cR+4)^{\sigma}} \right\rangle\\
&\quad -\frac{(2-\sigma)(\sigma-1)}{(\cR+4)^{\sigma+2}}|\Ric|^2 |\nabla \cR|^2 +\frac{4(T-L)(\cR+4)-2\sigma |\Ric|^4 }{(\cR+4)^{\sigma+1}} .
\end{split}
\end{equation}
On the other hand,
\begin{equation}
\begin{split}
&\rheat  (\cR+4)^{2-\sigma}\\
&=2(2-\sigma) (\cR+4)^{1-\sigma} |\Ric|^2 -(2-\sigma)(1-\sigma) (\cR+4)^{-\sigma}|\nabla \cR|^2\\
&=\frac{2(2-\sigma)}{(\cR+4)^{1+\sigma}} (\cR+4)^{2} |\Ric|^2 +\frac{2(\sigma-1)}{\cR+4} \left\langle\nabla
\cR, \nabla (\cR+4)^{2-\sigma} \right\rangle\\
&\quad -\frac{(2-\sigma)(\sigma-1)}{(\cR+4)^{2}}(\cR+4)^{2-\sigma} |\nabla \cR|^2.
\end{split}
\end{equation}
%
Therefore, the function $f$ satisfies
\begin{equation}
\begin{split}
&\quad \rheat f\\
&=\rheat \left[\frac{|\Ric|^2}{(\cR+4)^\sigma}-\frac13 (\cR+4)^{2-\sigma} \right] \\
&=\frac{2(\sigma-1)}{\cR+4} \langle\nabla \cR, \nabla f\rangle-\frac{(2-\sigma)(\sigma-1)}{(\cR+4)^{2}}|\nabla \cR|^2f\\
&\quad -\frac{2}{(\cR+4)^{\sigma+2}}|(\cR+4) \nabla_i \R_{jk}- \cR_i  \R_{jk}|^2\\
&\quad +\frac{2}{(\cR+4)^{\sigma+1}}\Big\{  (2-\sigma) |\Ric|^2\left[ |\Ric|^2-\frac13 (\cR+4)^2\right]-2J\Big\}
\end{split}
\end{equation}
where $J(\Ric):=|\Ric|^4+(\cR+4)(L-T)$. More generally we will regard $J$ as a function of 
$(0,2)$ tensors, depending only on their eigenvalues, that has this expression when applied to $\Ric$.

We now examine the evolution of $f$ wherever $f>0$. In this case,
\begin{equation}\label{LemmaStep: almostpreserveRRic-1}
\begin{split}
 \rheat f&\leq \frac{\sigma-1}{2-\sigma} \frac{|\nabla f|^2}{f}\\
 &\quad +\frac{2}{(\cR+4)^{\sigma+1}}\Big\{  (2-\sigma) |\Ric|^2 (\cR+4)^\sigma f-2J\Big\}.
\end{split}
\end{equation}
%
To control the final term, 
consider the modified Ricci curvature tensor $\widetilde\Ric=\Ric+  2g$ and the corresponding modified scalar curvature tensor $\tilde \cR=\cR+6$. 
By the almost pinching assumption for $\Ric$ and the fact that $\e\in (0,\frac1{100})$, we have 
$$\tilde \cR>0 \quad \text{and}\quad \widetilde\Ric \geq  \e \tilde \cR$$
on $B_t(x_0,1)$ for $t\leq   T$.

Consider the fourth order operator of $\Ric$ given by 
$$P(\Ric):=|\Ric|^4+\cR (L-T)$$
so that $J(\Ric)=P(\Ric)+4(L-T)$. Here $P(\Ric)$ is the quantity considered 
by Hamilton in \cite[Lemma 10.7]{Hamilton1982}. We instead apply it to the modified Ricci tensor $\widetilde{\Ric}$ and $\tilde\cR$.  By putting $\Ric=\widetilde{\Ric}-2g$ and $\cR=\tilde\cR-6$, and treating terms that are less than quartic as error terms, we compute
\begin{equation}
\begin{split}
&\quad (2-\sigma) |\Ric|^2 \left[|\Ric|^2-\frac13 ( \cR+4)^2 \right]-2J(\Ric)\\
&= (2-\sigma) |\widetilde\Ric-2g|^2 \left[|\widetilde\Ric-2g|^2-\frac13  (\tilde \cR-2)^2 \right]-2J(\widetilde{\Ric}-2g)\\
&\leq (2-\sigma)|\widetilde{\Ric}|^2\left[|\widetilde\Ric|^2-\frac13  \tilde \cR^2 \right]-2P(\widetilde{\Ric})+C_0(|\widetilde\Ric|^3+1)\\
\end{split}
\end{equation}
for some constant $C_0>0$.  Here we have used the fact that $T(\Ric)$ and $L(\Ric)$ are third order operators on $\Ric$.  Since $\widetilde\Ric>0$, we have $|\widetilde\Ric|\leq \tilde\cR$ and hence  \eqref{LemmaStep: almostpreserveRRic}  implies that 
\begin{equation}\label{LemmaStep: almostpreserveRRic-2}
\begin{split}
&\quad (2-\sigma) |\Ric|^2 \left[|\Ric|^2-\frac13 ( \cR+4)^2 \right]-2J(\Ric)\\
& \leq  (2-\sigma)|\widetilde{\Ric}|^2\left[|\widetilde\Ric|^2-\frac13  \tilde \cR^2 \right]-2P(\widetilde{\Ric})+C_1(\cR+4)^3.
\end{split}
\end{equation}
Applying \cite[Lemma 10.7]{Hamilton1982} on $\widetilde\Ric$, if $2-\sigma<\e^2$, then
\begin{equation}
\begin{split}
&\quad  (2-\sigma) |\widetilde\Ric|^2 \left[|\widetilde\Ric|^2-\frac13  \tilde \cR^2 \right]-2P(\widetilde\Ric) \\
&\leq -\e^2|\widetilde\Ric|^2 \left[|\widetilde\Ric|^2-\frac13  \tilde \cR^2 \right]\\
&\leq -\frac13 \e^2 (\cR+4)^{2+\sigma} f+C_2(\cR+4)^3
\end{split}
\end{equation}
for some $C_2>0$, where we have used the identity
$|\Ric|^2\geq \frac13 \cR^2$, and 
where as usual we pay close attention to quartic terms and treat lower order terms as errors.
Now we apply the observation in \cite[Page 539]{BrendleHuiskenSinestrari2011} to see that $f^\frac{1}{2-\sigma}\leq c_0^{-1} (\cR+4)$ for some $c_0(\e)>0$ since 
\begin{equation}
\begin{split}
(\cR+4)^\sigma f=|\Ric|^2-\frac13 (\cR+4)^2 &\leq |\Ric|^2\\
&= |\widetilde\Ric-2g|^2\\
&\leq C_3(\cR+4)^{2-\sigma} \cdot (\cR+4)^{\sigma}.
\end{split}
\end{equation}
Here we have used \eqref{LemmaStep: almostpreserveRRic} and the fact that $\widetilde\Ric>0$. Therefore,  we conclude that 
\begin{equation}
\label{LemmaStep: almostpreserveRRic-3}
\begin{split}
&  \quad(2-\sigma) |\widetilde\Ric|^2 \left[|\widetilde\Ric|^2-\frac13  \tilde \cR^2 \right]-2P(\widetilde\Ric)\\ 
&\leq -\frac{c_1}{2} (\cR+4)^{1+\sigma}  f^\frac{3-\sigma}{2-\sigma}+C_2(\cR+4)^3
\end{split}
\end{equation}
for some $c_1(\e)>0$.

\bigskip

Combining \eqref{LemmaStep: almostpreserveRRic-1}, \eqref{LemmaStep: almostpreserveRRic-2} and \eqref{LemmaStep: almostpreserveRRic-3}, we conclude that if $2-\e^2<\sigma<2$, then
\begin{equation}
\begin{split}
\rheat f\leq \frac{\sigma-1}{2-\sigma} \frac{|\nabla f|^2}{f}-c_2 f^\frac{3-\sigma}{2-\sigma} +C_4(\cR+4)^{2-\sigma}
\end{split}
\end{equation}
whenever $f>0$. Equivalently, $F=f^\frac{1}{2-\sigma}$ satisfies
\begin{equation}
\begin{split}
\rheat F&\leq -c_3 F^2 +C_5(\cR+4)^{2-\sigma} F^{\sigma-1}\\
&\leq -\frac12 c_3 F^2  +C_6 (1+(at^{-1})^{\frac{2(2-\sigma)}{3-\sigma}} )
\end{split}
\end{equation}
whenever $F(x,t)>0$. Here we have used the $a/t$ curvature decay assumption.

If we choose $\sigma$ sufficiently close to $2$ so that
$$\frac{2-\sigma}{3-\sigma}<\frac14,$$
then if $S_2<a^{-1}$, (recall $a>1\geq S_2\geq t$ also), the function $G=tF$ will satisfy
\begin{equation}
\label{3lines}
\begin{split}
t\cdot \rheat G&\leq t^2\cdot \left(-\frac12 c_3 F^2  +C_6 (1+(at^{-1})^{1/2} ) \right) +tF\\
&\leq -\frac12 c_3 G^2 +C_6\left(1+(at^3)^{1/2}\right)+G\\
&\leq -\frac14 c_3 G^2 +C_7
\end{split}
\end{equation}
for $t\in (0,S_2\wedge T]$ whenever $G>0$.

\bigskip

Now we are ready to apply the local maximum principle. Let $\phi$ be a non-increasing smooth function on $[0,+\infty)$ such that $\phi(s)\equiv 1$ on $[0,\frac34]$, vanishing outside $[0,1]$ and satisfying
$$\phi''\geq -10^3\phi\quad\text{and}\quad  |\phi'|^2\leq 10^3\phi.$$

Let $\Phi(x,t)=e^{-10^3t}\phi(d_t(x,x_0)+\b \sqrt{at})$ be a cutoff function on $M$. The constant $\b$ is a universal constant so that
\begin{equation}
\rheat \Phi \leq 0
\end{equation}
in the barrier sense using  \cite[Lemma 8.3]{Perelman2002}. We may assume $\Phi$ to be smooth when we apply the maximum principle; for example see \cite[Section 7]{SimonTopping2016} for a detailed exposition. 

Now we apply the maximum principle on $M\times [0,T\wedge S_2]$ to estimate $G$ on $B_t(x_0,\frac12)$.   Owing to the cutoff $\Phi$, the product $\Phi \cdot G$ must attain a maximum at some $(x_1,t_1)\in M\times [0,T\wedge S_2]$.  If $t_1=0$, the conclusion holds trivially since $G(x_1,0)=0$. Thus, we may assume $t_1>0$.  At $(x_1,t_1)$ where $\Phi G$ achieves a maximum, we may assume $G(x_1,t_1)>0$ so that
\begin{equation}
\label{Gbound}
\begin{split}
0&\leq  t\Phi \rheat (\Phi G)\\
&= tG \Phi \rheat \Phi+ \Phi^2 t\rheat G-2t \Phi\langle \nabla \Phi,\nabla G\rangle \\
&\leq  \Phi^2 \left( -\frac14 c_3 G^2+C_7\right)+2 G |\nabla \Phi|^2\\
&\leq -\frac18c_3 (\Phi G)^2+C_8.
\end{split}
\end{equation}
Here we have used that $\nabla (\Phi G)=0$ at $(x_1,t_1)$ and $t\leq S_2\leq 1$. 
We also used that $\frac{|\nabla \Phi|^2}{\Phi}$ is bounded by construction.

In particular we find  that $\Phi G(x_1,t_1)\leq C_9$ and hence, 
\begin{equation}
\sup_{M\times [0,T\wedge S_2]} \Phi G\leq \Phi G|_{(x_1,t_1)}\leq C_9.
\end{equation}

Notice that $B_t(x_0,\frac12)\subset \{(x,t):\Phi=1\}$ if $t<(16\b^2 a)^{-1}$.  This shows that if we shrink $S_2$ further, then for all $t\in (0,T\wedge S_2]$,
\begin{equation}
f(x_0,t)=\left(t^{-1}G(x_0,t)\right)^{2-\sigma}\leq C_{10}t^{\sigma-2}.
\end{equation}

To recover our conclusion from $f$,  it suffices to point out that from $\widetilde\Ric>0$ and \eqref{LemmaStep: almostpreserveRRic}, 
we have at $x_0$ that
\begin{equation}
\begin{split}
\left|\Ric-\frac13 \cR g\right|^2&\leq \left|\Ric-\frac13( \cR+4) g\right|^2+C_{11} (\cR+4)\\
&=(\cR+4)^{\sigma }f  +C_{12} (\cR+4)\\
&\leq C_{10}t^{\sigma-2}(\cR+4)^{\sigma } +C_{12} (\cR+4).
\end{split}
\end{equation}
Since $\sigma\in (1,2)$, we can absorb the final term, after possibly reducing $S_2$ one final time, thus completing the proof.
\end{proof}

\bigskip

Next, we show that under the estimate from Lemma~\ref{SecEst-pinching-Implies-spherical}, the Ricci flow gains a curvature decay estimate.


\begin{lma}\label{SecEst-pinching-curvatureEstimate}
Suppose $M^3$ is a non-compact (connected) manifold and $g(t),t\in[0,T]$ is a smooth solution to the Ricci flow on $M$ so that for some $x_0\in M$, 
\begin{enumerate}
\item[(i)] $B_t(x_0,1)\Subset M$ for $t\in[0,T]$;
\item[(ii)] $ \cR(x,t)>-4$ for $x\in B_t(x_0,1)$, $t\in [0,T]$;
\item[(iii)] $\left|\Ric-\frac13 \cR g\right|^2\leq Lt^{\sigma-2}  (\cR+L)^\sigma$ on $B_t(x_0,1)$ for $t\in [0,T]$ and some $L>0,\sigma\in (1,2)$.
\end{enumerate}
Then there exist $a_1(\sigma,L),S_3(\sigma,L)>0$ such that for $t\in (0,T\wedge S_3]$ 
we have
$$|\Rm(x_0,t)|\leq a_1t^{-1}.$$
\end{lma}

\begin{proof}
The proof of the curvature estimate follows the point-picking argument
from \cite[Lemma 2.1]{SimonTopping2016}.

Suppose the conclusion is false, then 
there exist $\sigma\in (1,2)$ and $L>0$ such that 
for any $a_k\rightarrow +\infty$, there exists a sequence of three-dimensional non-compact manifolds $M_k$, Ricci flows $g_k(t)$, $t\in [0,T_k]$ on $M_k$ and points $x_k\in M_k$ 
satisfying the hypotheses of the lemma but 
so that the curvature conclusion fails in an arbitrarily short time. We may assume $a_kT_k\rightarrow 0$. By smoothness of each Ricci flow, we can choose $t_k\in (0,T_k]$ so that
\begin{enumerate}
\item[(i)] $B_{g_k(t)}(x_k,1)\Subset M_k$ for $t\in [0,t_k]$;
\item[(ii)] $\cR(g_k(t))> -4$ on $B_{g_k(t)}(x_k,1)$, $t\in [0,t_k]$;
\item[(iii)]$\left|\Ric(g_k(t))-\frac13 \cR(g_k(t))g_k(t)\right|^2\leq Lt^{\sigma-2}  (\cR(g_k(t))+L)^\sigma$ \\ on $B_{g_k(t)}(x_k,1)$ for $t\in [0,t_k]$;
\item[(iv)]$|\Rm(g_k(t))|(x_k)<a_k t^{-1}$ 
for $t\in (0,t_k)$;
\item[(v)]$|\Rm(g_k(t_k))|(x_k)=a_k t_k^{-1}$.
\end{enumerate}


\bigskip

By (iv) and the fact that $a_kt_k\rightarrow 0$, \cite[Lemma 5.1]{SimonTopping2016} implies that for $k$ sufficiently large, we can find $\b>0,\tilde t_k\in (0,t_k]$ and $\tilde x_k\in B_{g_k(\tilde t_k)}(x_k,\frac{3}{4}-\frac{1}{2}\b\sqrt{a_k \tilde t_k})$ such that 
\begin{equation}\label{CURVATURE}
|\Rm(g_k(x,t))|\leq 4|\Rm(g_k(\tilde x_k,\tilde t_k))|=4Q_k
\end{equation}
whenever $d_{g_k(\tilde t_k)}(x,\tilde x_k)<\frac{1}{8}\b a_kQ_k^{-1/2}$ and $\tilde t_k-\frac{1}{8}a_kQ_k^{-1}\leq t\leq \tilde t_k$ where $\tilde t_k Q_k\geq a_k\rightarrow +\infty$.\bigskip

Consider the parabolic rescaling centred at $(\tilde x_k,\tilde t_k)$, namely $\tilde g_k(t)=Q_k g_k(\tilde t_k+Q_k^{-1}t)$ for $t\in [-\frac{1}{8}a_k,0]$ so that
\begin{enumerate}
\item[(a)] $|\Rm_{\tilde g_k(0)}(\tilde x_k)|=1$;
\item[(b)] $|\Rm_{\tilde g_k(t)}|\leq 4$ on $B_{\tilde g_k(0)}(\tilde x_k,\frac18 \b a_k)\times [-\frac18 a_k,0]$, and
\item[(c)]   on $B_{\tilde g_k(0)}(\tilde x_k,\frac18 \b a_k)$,
\begin{equation}
\left\{
\begin{array}{ll}
\displaystyle\cR(\tilde g_k(0))>-4Q_k^{-1};\\
\displaystyle \left|\Ric(\tilde g_k(0))-\frac13 \cR(\tilde g_k(0))\tilde g_k(0)\right|^2 \leq  La_k^{\sigma-2}( \cR(\tilde g_k(0))+LQ_k^{-1})^\sigma,
\end{array}
\right.
\end{equation}
\end{enumerate}
where we used that $\tilde t_k Q_k\geq a_k$ to obtain the final estimate.

It what follows, we tread carefully to accommodate the fact that we have no uniform lower bound on the injectivity radius of $\tilde g_k(0)$. 
By (b), we can pick a universal $\rho>0$ so that the conjugate radius 
is always larger than $\rho$.
Therefore we can lift $\left(B_{\tilde g_k(0)}(\tilde x_k,\rho),\tilde g_k(t) \right)$ to $(B_{euc}(\rho),\tilde g_k(t))$ by the exponential map of $\tilde g_{k}(0)$ at $\tilde x_k$. By Shi's estimate \cite{Shi1989}, we may assume $\tilde g_k(t)$ converges to $\tilde g_\infty(t)$  uniformly locally smoothly on the $B_{euc}(\rho)\times (-\infty,0]$
after passing to a subsequence. Using (c), we conclude that
\begin{equation}
\left\{
\begin{array}{ll}
\cR(\tilde g_\infty(0))&\geq 0;\\
\Ric(\tilde g_\infty(0))&=\frac13 \cR(\tilde g_\infty(0))\tilde g_\infty(0)
\end{array}
\right.
\end{equation}
and hence $\tilde g_\infty(0)$ is Einstein with non-negative Einstein constant on the Euclidean ball of radius $\rho$, thanks to the Bianchi identity,
and thus has constant non-negative sectional curvature.
By (a) and the smooth convergence, the constant sectional curvature must be positive and hence $\Ric(\tilde g_\infty(0))=3\a$ on the Euclidean ball of radius $\rho$ for some  
$\a>0$ that could be computed explicitly.

We now extend this control to a larger region. For any fixed $r>\rho$ and each $k$, 
pick a maximal disjoint collection of balls $B_{\tilde g_k(0)}(y^j_k,\frac{\rho}{100})$
within $B_{\tilde g_k(0)}(\tilde x_k,r)$,
indexed by $j$. By volume comparison, the number of points $y^j_k$ is bounded uniformly in $k$, and so by passing to a subsequence we may assume that $j$ always ranges from $1$ to some $J\in\mathbb{N}$. We thus have $J$ sequences $y^j_k$ that can be substituted for $\tilde x_k$
in the previous process of lifting by the exponential map and passing to a subsequence to give a limit that is of constant sectional curvature.
After passing to all $J$ subsequences, because of the overlaps between the covering balls
$B_{\tilde g_k(0)}(y^j_k,\rho)$ we deduce that each limit has the \emph{same} constant sectional curvature.
In particular, after deleting finitely many elements of the sequence
we may assume that 
$\Ric(\tilde g_k(0))> 2 \a$ on $B_{\tilde g_k(0)}(\tilde x_k,r)$ for all $k\in \mathbb{N}$.

But it is well understood that one cannot have a large region of a manifold with controllably positive Ricci curvature:
Set $L:=2\pi/\sqrt{\a}$ and take the subsequence constructed above corresponding to 
$r=2L$. 
Then  $\Ric(\tilde g_{1}(0))> 2\a$ on $B_{\tilde g_{1}(0)}(\tilde x_{1},2L)$. Suppose there exists $z\in B_{\tilde g_{1}(0)}(\tilde x_{1},2L)$ such that $d_{\tilde g_{1}(0)}(\tilde x_{1},z)=L$. Let $\gamma$ be a minimizing geodesic realizing $d_{\tilde g_{1}(0)}(\tilde x_{1},z)=L$; then $\gamma$ must lie inside $B_{\tilde g_{1}(0)}(\tilde x_{1},2L)$ where the Ricci curvature is bounded from below by $2\a$. 
By the proof of the Bonnet-Myers theorem, the Ricci curvature's lower bound implies $d_{\tilde g_{1}(0)}(\tilde x_{1},z)\leq \pi/\sqrt{\a}$ which is impossible. And hence, $B_{\tilde g_{1}(0)}(\tilde x_{1},L)=M_{1}$ which  contradicts that $B_{g_{1}(\tilde t_{1})}( x_{1},1)\Subset M_{1}$ and the non-compactness of $M_{1}$.
\end{proof}

\section{Ricci flow lemmas}

In this section we recall some useful results about Ricci flow. The first of these is a result of Hochard that allows us to modify an incomplete Riemannian metric at its extremities in order to make it complete, without increasing the curvature too much, and without changing it in the interior.

\begin{lma}[{Hochard \cite[{Corollaire IV.1.2}]{hochard_thesis}}]
There exists $\sigma(n)>1$ such that given a Riemannian manifold 
$(N^n,g)$ with  $|\Rm(g)|\leq \rho^{-2}$ throughout for some $\rho>0$, 
there exists a complete Riemannian metric $h$ on $N$ such that 
\begin{enumerate}
\item
$h\equiv g$ on $N_\rho:=\{x\in N: B_g(x,\rho)\Subset N\}$, and
\item
$|\Rm(h)|\leq \sigma\rho^{-2}$ throughout $N$.
\end{enumerate}
\end{lma}

An immediate consequence of Hochard's lemma, via Shi's existence theorem for Ricci flow starting with complete initial metrics of bounded curvature \cite{Shi1989}, is the 
following local existence result for Ricci flow.
\begin{prop}\label{construction-local-RF-new}
There exist constants $\alpha(n)\in (0,1]$ and $\Lambda(n)>1$ so that the following is true.
Suppose $(N^n,h_0)$ is a smooth manifold (not necessarily complete)
that satisfies $|\Rm(h_0)|\leq \rho^{-2}$ throughout, for some $\rho>0$.
Then there exists a smooth Ricci flow $h(t)$ on $N$
for $t\in [0,\alpha \rho^2]$, with the properties that 
\begin{enumerate}
\item[(i)] $h(0)=h_0$ on $N_\rho=\{x\in N: B_{h_0}(x,\rho)\Subset N\}$;
\item[(ii)] $|\Rm(h(t))|\leq  \Lambda \rho^{-2}$ throughout $N\times [0,\alpha \rho^2]$.
\end{enumerate}
\end{prop}
When applying this lemma, we will only be interested in the restriction of $h(t)$ to the subset $N_\rho$.



We also recall the shrinking balls lemma,
which is one of the local ball inclusion results 
based on  the distance distortion estimates of Hamilton and Perelman from  \cite[Lemma 8.3]{Perelman2002}.
\begin{lma}[{\cite[Corollary 3.3]{SimonTopping2016}}]
\label{l-balls}
 There exists a constant $\beta=\beta(n)\geq 1$ depending only on $n$ such that the following is true. Suppose $(N^n,g(t))$ is a Ricci flow for $t\in [0,S]$ and $x_0\in N$   with $B_0(x_0,r)\Subset N$ for some $r>0$, and $\Ric(g(t))\leq a/t$ on $B_{g_0}(x_0,r)$ for each $t\in (0,S]$. Then
$$B_t\left(x_0,r-\beta\sqrt{a t}\right)\subset B_{g_0}(x_0,r).$$
\end{lma}

\section{Existence of Ricci flow}

In this section, we will construct the Ricci flow starting from complete non-compact manifolds with non-negatively pinched Ricci curvature. More generally, we will construct the partial Ricci flow which is based on the ideas from the work of Hochard \cite{Hochard2016}, Simon and the second named author \cite{SimonTopping2021}.
\begin{thm}\label{Thm:Construct-partialRF}
For all $\frac1{12}>\e>0$, there exist $T(\e),a_0(\e)>0$ such that the following holds.
Suppose $(M^3,g_0)$ is a three-dimensional manifold and $p\in M$ so that
\begin{enumerate}
\item[(i)] $B_{g_0}(p,R+4)\Subset M$ for some $R\geq1$; 
\item[(ii)]$\Ric(g_0)\geq \e \cR(g_0)\geq 0$ on $B_{g_0}(p,R+4)$.
\end{enumerate}
Then there exists a smooth Ricci flow solution $g(t)$ defined on $B_{g_0}(p,R)\times [0,T]$ such that $g(0)=g_0$, $|\Rm(x,t)|\leq a_0t^{-1}$ and 
\begin{equation}
\label{ape}
\Ric(x,t)\geq \e\cR(x,t)-1
\end{equation}
for all $(x,t)\in B_{g_0}(p,R)\times (0,T]$.
\end{thm}

\begin{proof}
We begin by specifying the constants that will be used in the proof. 
For the given pinching constant $\e$, we will use:
\begin{itemize}
\item $\Lambda(3)>1$ and $\a(3)\in (0,1]$:  from Proposition~\ref{construction-local-RF-new}; 
\item $\b(3)$: from Lemma~\ref{l-balls};
\item $L_1(\e)>0,\sigma(\e)\in (1,2)$: from Lemma~\ref{SecEst-pinching-Implies-spherical};
\item $a_1(\sigma,L_1),S_3(\sigma,L_1)$: from Lemma~\ref{SecEst-pinching-curvatureEstimate};
\item $a_0:=\max\{\Lambda \a, \Lambda (a_1+\a),1\}$;
\item $S_1(a_0,\e)$: from Lemma~\ref{SecEst-almost-pinching};
\item $S_2(a_0,\e)$: from Lemma~\ref{SecEst-pinching-Implies-spherical};
\item $\Lambda_0:=16\cdot \max\{ (a_0a_1)^{-1/2},\b, S_1^{-1/2}, S_2^{-1/2}, S_3^{-1/2}\}$;
\item $\mu=\sqrt{1+\a a_1^{-1}}-1>0$.
\end{itemize}
\bigskip

Choose $1\geq \rho>0$ sufficiently small so that for all $x\in B_{g_0}(p,R+4)$, we have $|\Rm(g_0)|\leq \rho^{-2}$.  By Proposition~\ref{construction-local-RF-new},
applied with $N=B_{g_0}(p,R+4)$, 
we can find a smooth solution $g(t)$ to the Ricci flow on (a superset of)
$B_{g_0}(p,R+3)\times [0,\a \rho^2]$ with $|\Rm(g(t))|\leq \Lambda\rho^{-2}$
and (once restricted to $B_{g_0}(p,R+3)$) with
initial data $g_0$.  
Because $ a_0\geq \Lambda \a$, the curvature bound can be weakened to
$|\Rm(g(t))|\leq a_0 t^{-1}$.

\bigskip

We now define sequences of times $t_k$ and radii $r_k$ inductively as follows:
\begin{enumerate}
\item[(a)] $t_1=\a \rho^2,\;r_1=R+3$, where $\rho$ is obtained from above;
\item [(b)] $t_{k+1}=(1+\mu)^2 t_k=(1+\a a_1^{-1})t_k $ for $k\geq 1$;
\item [(c)] $r_{k+1}=r_k-\Lambda_0 \sqrt{a_0 t_k}$ for $k\geq 1$.
\end{enumerate}

Let $\mathcal{P}(k)$ be the following statement: there is a smooth Ricci flow solution $g(t)$ defined on $B_{g_0}(p,r_k)\times [0,t_k]$ with $g(0)=g_0$ such that $|\Rm(g(t))|\leq a_0t^{-1}$. Noted that we choose $\rho$ so that $\mathcal{P}(1)$ is true. Our goal is to show that $\mathcal{P}(k)$ is true for all $k$ so that $r_k>0$.
\bigskip

We now perform an inductive argument. Suppose $\mathcal{P}(k)$ is true, that is to say that $\exists g(t)$ on $B_{g_0}(p,r_k)\times [0,t_k]$ with $|\Rm(g(t))|\leq a_0t^{-1}$. We want to show that $\mathcal{P}(k+1)$ is true provided that $r_{k+1}>0$.\bigskip

Let $x\in  B_{g_0}(p,r_{k+1}+\frac12 \Lambda_0 \sqrt{a_0t_k})$ so 
\begin{equation}
B_{g_0}\left(x,\frac14 \Lambda_0 \sqrt{a_0t_k}\right)\Subset B_{g_0}(p,r_k).
\end{equation}
Consider the rescaled Ricci flow $\tilde g(t)=\lambda_1^{-2} g(\lambda_1^{2}t),t\in [0,16\Lambda_0^{-2}a_0^{-1}]$ where $\lambda_1=\frac14 \Lambda_0 \sqrt{a_0t_k}$ so that $B_{g(0)}(x,\frac14 \Lambda_0 \sqrt{a_0t_k})=B_{\tilde g(0)}(x,1)$. Since $\Ric(\tilde g(0))\geq \e \cR(\tilde g(0))$ on  $B_{\tilde g(0)}(x,1)$ and $|\Rm(\tilde g(t))|\leq a_0t^{-1}$ on $B_{\tilde g(0)}(x,1)$ for $t\in (0,16\Lambda_0^{-2}a_0^{-1}]$, we can apply Lemma~\ref{SecEst-almost-pinching} to $\tilde g(t)$ to conclude that

$$\Ric(\tilde g(x,t))\geq \e \cR(\tilde g(x,t))-1$$
for $t\in [0,16\Lambda_0^{-2}a_0^{-1}]$ since 
$ 16\Lambda_0^{-2}\leq S_1$ and $a_0\geq 1$. 
Rescaling back to $g(t)$, we see that on $B_{g_0}\left(p,r_{k+1}+\frac12 \Lambda_0 \sqrt{a_0t_k}\right)\times [0,t_k]$, we have
\begin{equation}\label{intermediate-bound-1}
\Ric(g(t))\geq \e \cR(g(t))-\left(\frac14 \Lambda_0 \sqrt{a_0t_k}\right)^{-2}.
\end{equation}

Next we aim to prove the estimate \eqref{intermediate-bound-2} below on the ball
$B_{g_0}(p,r_{k+1}+\frac14\Lambda_0 \sqrt{a_0t_k})$.
Fix $x\in B_{g_0}(p,r_{k+1}+\frac14\Lambda_0 \sqrt{a_0t_k})$.
By the shrinking balls lemma \ref{l-balls} and the assumption 
$\Lambda_0\geq 8\b$, we 
have the inclusions
$$\textstyle B_{g(t)}\left(x,\frac1{8} \Lambda_0\sqrt{a_0t_k}\right)
\subset B_{g_0}\left(x,\frac14 \Lambda_0\sqrt{a_0t_k}\right)\Subset 
B_{g_0}\left(p,r_{k+1}+\frac12\Lambda_0 \sqrt{a_0t_k}\right)$$
for all $t\in [0,t_k]$, where \eqref{intermediate-bound-1} holds.
Consider the rescaled Ricci flow $\hat g(t)=\lambda_2^{-2}g(\lambda_2^2t),t\in [0,8^2\Lambda_0^{-2}a_0^{-1}]$ where $\lambda_2=\frac18\Lambda_0 \sqrt{a_0t_k}$ so that $B_{g(\lambda_2^{2}t)}(x,\frac18 \Lambda_0 \sqrt{a_0t_k})=B_{\hat g(t)}(x,1)$. 
Under the rescaling, the pinching estimate \eqref{intermediate-bound-1}
becomes 
$$\Ric(\hat g(t))\geq \e \cR(\hat g(t))- \frac14 \geq \e \cR(\hat g(t))-1.$$
Keeping in mind the $a_0/t$ curvature decay we can apply Lemma~\ref{SecEst-pinching-Implies-spherical} over the whole time interval $[0,8^2\Lambda_0^{-2}a_0^{-1}]$ because 
the assumptions 
$ 8^2\Lambda_0^{-2}\leq S_2$ and $a_0\geq 1$ imply that
$S_2\geq 8^2\Lambda_0^{-2}a_0^{-1}$.
Rescaling the conclusion of Lemma~\ref{SecEst-pinching-Implies-spherical}
back to $g(t)$ shows that 
\begin{equation} \label{intermediate-bound-2}
\left|\Ric-\frac13 \cR g\right|^2\leq \frac{L_1}{t^{2-\sigma}} \cdot \left[\cR+4\left( \frac18 \Lambda_0\sqrt{a_0t_k}\right)^{-2}\right]^\sigma
\end{equation}
on $B_{g_0}(p,r_{k+1}+\frac14 \Lambda_0 \sqrt{a_0t_k})\times [0,t_k]$ as required.

\bigskip

Next we would like to establish $a_1/t$ decay of curvature on the ball
$B_{g_0}(p,r_{k+1}+\frac18 \Lambda_0\sqrt{a_0t_k})$ for all $t\in (0,t_k]$. 
Pick an arbitrary point $x\in B_{g_0}(p,r_{k+1}+\frac18 \Lambda_0\sqrt{a_0t_k})$ 
where we would like this decay to hold. 
By the shrinking balls lemma \ref{l-balls} and the assumption 
$\Lambda_0\geq 16\b$, we 
have the inclusions
$$\textstyle B_{g(t)}\left(x,\frac1{16} \Lambda_0\sqrt{a_0t_k}\right)
\subset B_{g_0}\left(x,\frac18 \Lambda_0\sqrt{a_0t_k}\right)\Subset 
B_{g_0}\left(p,r_{k+1}+\frac14\Lambda_0 \sqrt{a_0t_k}\right)$$
for all $t\in [0,t_k]$. 
Thus rescaling the flow to 
$\lambda_3^{-2}g(\lambda_3^2t)$, where $\lambda_3=\frac1{16} \Lambda_0\sqrt{a_0t_k}$,
which lives for $t\in [0,(16)^2\Lambda_0^{-2}a_0^{-1}]$, we can 
apply Lemma~\ref{SecEst-pinching-curvatureEstimate}, using 
\eqref{intermediate-bound-2} and \eqref{intermediate-bound-1},
to deduce that 
\begin{equation}\label{intermediate-bound-3}
|\Rm(x,t)|\leq a_1t^{-1}
\end{equation}
for all $t\in (0,t_k]$ provided that $S_3\geq (16)^2\Lambda_0^{-2}a_0^{-1}$, and this indeed holds by our assumptions that $ 16^2\Lambda_0^{-2}\leq S_3$ and  $a_0\geq 1$.

\bigskip

Denote $\Omega=B_{g_0}\left(p,r_{k+1}+\frac18\Lambda_0 \sqrt{a_0t_k}\right)$ so that for $h_0=g(t_k)$, estimate \eqref{intermediate-bound-3} gives $\sup_{\Omega}|\Rm(h_0)|\leq \rho^{-2}$ 
where $\rho=\sqrt{t_ka_1^{-1}}$. 
Moreover, for $x\in B_{g_0}(p,r_{k+1})$, 
the assumptions 
$ \Lambda_0\sqrt{a_0 a_1}\geq 16$ and $\Lambda_0\geq 16\beta$,
together with the shrinking balls lemma \ref{l-balls} (needing only the weaker $a_0/t$ decay rather than the $a_1/t$ decay that we have proved) give 
\begin{equation}
B_{g(t_k)}(x,\rho)\subset B_{g(t_k)}\left(x, \frac1{16}\Lambda_0\sqrt{a_0t_k}\right)
\subset B_{g_0}\left(x,\frac18\Lambda_0\sqrt{a_0t_k}\right)\Subset \Omega.
\end{equation}
This shows that $B_{g_0}(p,r_{k+1})\subset \Omega_{\rho}$, 
where $\Omega_{\rho}$ is computed with respect to $g(t_k)$. 
Hence, we may apply Proposition~\ref{construction-local-RF-new} to find a Ricci flow $g(t)$ on (a superset of) $B_{g_0}(p,r_{k+1})\times [t_k,t_k+\a \rho^2]$,
extending the existing $g(t)$ on this smaller ball, with 
\begin{equation}
|\Rm(g(t))|\leq \Lambda \rho^{-2}= \Lambda a_1t_k^{-1}\leq a_0t^{-1}
\end{equation}
since $ \Lambda (a_1+\a)\leq a_0$ and $t_k+\a \rho^2=t_k\left(1+\a a_1^{-1}\right)=t_{k+1}$. This shows that $\mathcal{P}(k+1)$ is true provided that $r_{k+1}>0$.

\bigskip

Since $\lim_{j\to +\infty}r_j=-\infty$, there is $i\in \mathbb{N}$ such that $r_i\geq R+1$ and $r_{i+1}<R+1$. In particular, $\mathcal{P}(i)$ is true since $r_i>0$. We now estimate $t_i$.
\begin{equation}
\begin{split}
R+1>r_{i+1}
&=r_1- \Lambda_0 \sqrt{a_0}\cdot \sum_{k=1}^i \sqrt{t_k} \\
&\geq R+3- \Lambda_0 \sqrt{a_0t_i}\cdot\sum_{k=0}^{\infty} (1+\mu)^{-k} \\
&=R+3- \sqrt{t_i}\cdot \frac{\Lambda_0 \sqrt{a_0}(1+\mu)}{\mu}.
\end{split}
\end{equation}

This implies 
\begin{equation}
t_i >\frac{4\mu^2 }{a_0 \Lambda_0^2(1+\mu)^2}=:T(\e).
\end{equation}
In other words, there exists a smooth Ricci flow solution $g(t)$ defined on $B_{g_0}(p,R+1)\times [0,T]$ so that $g(0)=g_0$ and $|\Rm(g(t))|\leq a_0t^{-1}$. 
That the almost pinching estimate \eqref{ape} holds at an arbitrary point
$x_0\in B_{g_0}(p,R)$ follows immediately from an application of
Lemma~\ref{SecEst-almost-pinching} on $B_{g_0}(x_0,1)$ provided we allow ourselves 
to reduce $T>0$. This completes the proof.
\end{proof}

\bigskip

By an exhaustion argument, we can prove Theorem~\ref{SecIntro-RFexist} now. 
\begin{proof}[Proof of Theorem~\ref{SecIntro-RFexist}]

By reducing $\e>0$ if necessary, we may assume that $\e\in (0,\frac{1}{100})$.
 Let $R_i\to +\infty$ and denote $h_{i,0}=R_i^{-2}g_0$ so that $\Ric(h_{i,0})\geq \e \cR(h_{i,0})$ on $M$. By Theorem~\ref{Thm:Construct-partialRF}, there is a Ricci flow solution $h_i(t)$ on $B_{h_{i,0}}(p,1)\times [0,T]$ with 
\begin{enumerate}
\item[(a)] $|\Rm(h_i(t))|\leq a_0t^{-1}$;
\item [(b)] $\Ric(h_i(t))\geq \e \cR(h_i(t))-1$.
\end{enumerate}

Define $g_i(t)=R_i^2 h_i(R_i^{-2}t)$ which is a Ricci flow solution on $B_{g_0}(p,R_i)\times [0,TR_i^2]$  with 
\begin{equation}
\left\{
\begin{array}{ll}
g_i(0)=g_0; \\
|\Rm(g_i(t))|\leq a_0 t^{-1};\\
\Ric(g_i(t))\geq \e \cR(g_i(t))-R_i^{-2}
\end{array}
\right.
\end{equation}
on each $B_{g_0}(p,R_i)\times (0,TR_i^2]$.

By \cite[Corollary 3.2]{Chen2009} (see also \cite{Simon2008}) and a modification of Shi's higher order estimate given in \cite[Theorem 14.16]{ChowBookII}, we infer that for any $k\in \mathbb{N}$, $S>0$ and $\Omega\Subset M$, we can find $C(k,\Omega,g_0,\e,S)>0$ so that for sufficiently large $i$ we have
\begin{align}
\sup_{\Omega\times [0,S]}|\nabla^k \mathrm{Rm}(g_i(t))|\leq C(k,\Omega,g_0,\e,S).
\end{align}
By working in coordinate charts and  applying the Ascoli-Arzel\`a Theorem, we may pass to a subsequence to obtain a smooth solution $g(t)=\lim_{i\rightarrow +\infty}g_i(t)$ of the Ricci flow on $M\times [0,+\infty)$ so that $g(0)=g_0$, $|\mathrm{Rm}(x,t)|\leq a_0t^{-1}$ and
\begin{equation}
\Ric(x,t)\geq \e \cR(x,t)
\end{equation}
 for all $(x,t)\in M\times[0,+\infty)$.
 Moreover, it is a complete solution by Lemma \ref{l-balls}.   
By tracing this pinching estimate, we deduce that $\cR\geq 0$.
 This completes the proof.
\end{proof}

\end{document}